\documentclass{amsart}

\usepackage{amsmath, amssymb}

\theoremstyle{plain}
\newtheorem{theorem}{Theorem}[section]

\newtheorem{lemma}[theorem]{Lemma}
\newtheorem{proposition}[theorem]{Proposition}

\theoremstyle{remark}
\newtheorem*{remark}{Remark}

\theoremstyle{definition}

\numberwithin{figure}{section}

\newcommand{\DD}{{\mathbb D}}
\newcommand{\cA}{{\mathcal A}}
\newcommand{\cB}{{\mathcal B}}
\newcommand{\cC}{{\mathcal C}}
\newcommand{\cS}{{\mathcal S}}

\DeclareMathOperator{\hol}{Hol}

\begin{document}

\author{Philippe Drouin}
\address{D\'epartement de math\'ematiques et de statistique, Universit\'e Laval, Qu\'ebec (QC), Canada G1V 0A6}
\email{philippe.drouin.10@ulaval.ca}

\author{Thomas Ransford}
\address{D\'epartement de math\'ematiques et de statistique, Universit\'e Laval, Qu\'ebec (QC), Canada G1V 0A6}
\email{thomas.ransford@mat.ulaval.ca}
\thanks{TR supported by grants from NSERC and from the Canada Research Chairs program.}

\title{Covering numbers and schlicht functions}

\begin{abstract}
We determine upper and lower bounds for the minimal number of balls of a given radius needed to cover the space of schlicht functions.
\end{abstract}

\keywords{Covering number, holomorphic function, schlicht function}
\subjclass[2010]{30C45, 54E35}

\maketitle

\section{Introduction}

Let $(X,d)$ be a  metric space.
The \emph{$\delta$-covering number} of $X$, denoted $N_X(\delta)$, 
is the minimal number of closed balls of radius $\delta$ needed to cover  $X$.
It may happen that $N_X(\delta)=\infty$, but if $X$ is compact then always $N_X(\delta)<\infty$.
The quantity $N_X(\delta)$ arises in connection with the notion of \emph{box-counting dimension} of $X$, which is defined as
\[
\dim_B(X):=\lim_{\delta\to0^+}\frac{\log N_X(\delta)}{\log(1/\delta)},
\]
whenever the limit exists (see e.g.\ \cite[\S2.1]{Fa14}).
The covering number also plays an important role in the theory of universal approximation
(see e.g.\ \cite{KNR12}). In the latter context, the metric space $X$
is usually a function space. 

In this paper, our aim is to estimate
the covering number for a particular function space,
namely the class $\cS$ of \emph{schlicht functions}.
This is the class of all holomorphic functions $f$ on the open unit disk $\DD$ 
such that $f$ is injective and satisfies $f(0)=0$ and $f'(0)=1$. 
The class $\cS$ arises naturally in connection with the Riemann mapping theorem. It has been studied extensively and is the 
subject of several books (e.g.\ \cite{Du83,Go83,Po75,Sc75}).

\section{Statement of main result}

Let $\hol(\DD)$ denote the set of all holomorphic functions 
on the open unit disk~$\DD$. It is a Fr\'echet space with 
respect to the topology
of uniform convergence on compact subsets of $\DD$.
There are many choices of metric that give rise to this topology.
A very common one is to take
\begin{equation}\label{E:metric}
d(f,g):=\sum_{j\ge1} \lambda_j\min\Bigl\{1,\max_{|z|\le r_j}{|f(z)-g(z)|}\Bigr\},
\end{equation}
where $(\lambda_j)$ is a positive sequence with $\sum_j\lambda_j<\infty$, and $r_j$ is a sequence in $(0,1)$
such that $r_j\to1$.

With respect to this metric, the set $\cS$ is a compact subset 
of $\hol(\DD)$ (see e.g.\ \cite[p.276]{Du83}),
and therefore a compact metric space in its own right. Our aim is to estimate the covering number of this space. 
In order to do this, we need to impose some conditions on the sequences $(\lambda_j)$
and $(r_j)$. We shall assume that there exist constants
$0<\lambda<1$ and $\alpha>0$ such that
\begin{equation}\label{E:conditions}
\lambda_j\asymp \lambda^j
\quad\text{and}\quad
1-r_j\asymp j^{-\alpha}.
\end{equation}
These choices seem reasonable. In the literature,
$\lambda_j$ is nearly always taken to be $2^{-j}$.
Usually $r_j$ is not specified explicitly,
but in \cite[\S9.1]{Du83} it is taken to be $1-1/j$.

(Here and in what follows,  we write $a\lesssim b$ to signify that $a,b$ are positive functions or sequences
such that $a/b$ remains bounded above, and $a\asymp b$
 to signify that $a\lesssim b$ and $b\lesssim a$ hold together.)

The following theorem is our main result.

\begin{theorem}\label{T:main}
Let $d$ be the metric on $\cS$ given by \eqref{E:metric},
where  $(\lambda_j)$ and $(r_j)$ satisfy~\eqref{E:conditions}. Then
\begin{equation}\label{E:main}
\log^2(1/\delta)\lesssim \log N_\cS(\delta)\lesssim \log^{2+\alpha}(1/\delta)\quad (\delta\to0^+).
\end{equation}
\end{theorem}

The lower bound in \eqref{E:main}
is a generalization of a result obtained in \cite{KNR12}.
It shows, in particular, that the box-counting dimension of $\cS$ is infinite. We believe that the upper bound in \eqref{E:main} is new.

\section{Proof of the lower bound in Theorem~\ref{T:main}}\label{S:lower}

Our starting point is the following sufficient condition for membership of $\cS$. The result is well known, but we include a short proof for convenience.

\begin{proposition}\label{P:suff}
Let $f(z):=z+\sum_{k\ge2}a_kz^k$, where $\sum_{k\ge2}k|a_k|\le 1$.
Then $f\in\cS$.
\end{proposition}

\begin{proof}
We may suppose that at least one $a_k\ne0$,
otherwise the result is obvious.
Let $z,w$ be distinct points of $\DD$. 
Then, by the triangle inequality,
\[
|f(z)-f(w)|\ge |z-w|-\Bigl|\sum_{k\ge2}a_kz^k-\sum_{k\ge2}w^k\Bigr|.
\]
Now
\[
\Bigl |\sum_{k\ge2} a_kz^k-\sum_{k\ge2}a_kw^k\Bigr|
\le \sum_{k\ge2}|a_k||z^k-w^k|
<\sum_{k\ge2}k|a_k||z-w|\le|z-w|,
\]
the middle inequality being strict because at least one $a_k\ne0$.
It follows that $f(z)\ne f(w)$. Thus $f$ is injective, and so $f\in\cS$.
\end{proof}

Let $\cA$ be the family of  functions $f(z):=z+\sum_{k\ge2}a_kz^k$
such that $\sum_{k\ge2}k|a_k|\le1$.
By virtue of Proposition~\ref{P:suff}, we have $\cA\subset\cS$,
and therefore $N_\cS(\delta)\ge N_\cA(2\delta)$.
Thus, to establish the lower bound in Theorem~\ref{T:main},
it suffices to prove the following theorem.

\begin{theorem}\label{T:cA}
Let $d$ be the metric on $\cA$ given by \eqref{E:metric},
where  $(\lambda_j)$ and $(r_j)$ satisfy~\eqref{E:conditions}. Then
\begin{equation}\label{E:cA}
 \log N_\cA(\delta)\gtrsim \log^{2}(1/\delta)\quad (\delta\to0^+).
\end{equation}
\end{theorem}

To prove this result, it is first convenient to establish a lemma.

\begin{lemma}\label{L:cA}
Let $f,g\in \cA$, say $f(z):=z+\sum_{k\ge2}a_kz^k$ 
and $g(z):=z+\sum_{k\ge2}b_kz^k$. Then
\[
d(f,g)\ge\sum_{k\ge2}(\lambda_k r_k^k/2)|a_k-b_k|.
\]
\end{lemma}

\begin{proof}
The standard Cauchy estimate for Taylor coefficients gives that
\[
|a_k-b_k|\le \max_{|z|\le r_k}|f(z)-g(z)|/r^k.
\]
Also, since $|a_k|,|b_k|,r_k\le1$, we have $|a_k-b_k|r_k^k/2\le1$.
Combining these observations, we obtain
\begin{align*}
d(f,g)
&=\sum_{k\ge1}\lambda_k \min\Bigl\{1,\max_{|z|\le r_k}|f(z)-g(z)|\Bigr\}\\
&\ge\sum_{k\ge2}\lambda_k\min\Bigl\{1,|a_k-b_k|r_k^k\Bigr\}\\
&\ge\sum_{k\ge2}\lambda_k|a_k-b_k|r_k^k/2.\qedhere
\end{align*}
\end{proof}

\begin{proof}[Proof of Theorem~\ref{T:cA}]
Let $n\ge2$, and let $K$ be a large integer depending on $n$,
to be chosen later. Consider the collection of functions
\[
f(z):=z+\sum_{k=2}^n\frac{1}{kn}\frac{t_k}{K}z^k
\quad\Bigl(t_k\in\{1,2,\dots,K\},~k=2,\dots,n\Bigr).
\]
All of these functions belong to $\cA$.
There are $K^{n-1}$ of them.
Also, by Lemma~\ref{L:cA}, the $d$-distance between any two of them,
$f,g$ say, is at least
\[
d(f,g)\ge \Bigl(\min_{2\le k\le n}\frac{\lambda_kr_k^k}{2}\Bigr)\frac{1}{n^2K}.
\]
No ball of radius one third the right-hand side can contain more than one of these functions. It follows that
\begin{equation}\label{E:complicated}
N_{\cA}\Bigl(\frac{1}{3}\Bigl(\min_{2\le k\le n}\frac{\lambda_kr_k^k}{2}\Bigr)\frac{1}{n^2K}\Bigr)\ge K^{n-1}.
\end{equation}
By our assumptions, there exists $\alpha>0$ such that $1-r_k\asymp k^{-\alpha}$.
It follows that $r_k^k\ge e^{-ck^{1-\alpha}}$ for some  constant~$c>0$.
Since also $\lambda_k\ge C \lambda^k$ for some constant $C>0$, we have
\[
\frac{1}{3}\Bigl(\min_{2\le k\le n}\frac{\lambda_kr_k^k}{3}\Bigr)\frac{1}{n^2}
\ge \frac{C}{6}\exp\Bigl(-n\log(1/\lambda)-cn^{1-\alpha}-2\log n\Bigr).
\]
If we choose $\rho\in(0,\lambda)$, then 
$n\log(1/\lambda)+cn^{1-\alpha}+2\log n\le n\log(1/\rho)$ for all sufficiently large $n$,
and for these $n$ we then have
\begin{equation}\label{E:technical}
\Bigl(\min_{2\le k\le n}\frac{\lambda_kr_k^k}{6}\Bigr)\frac{1}{n^2}\ge\rho^n.
\end{equation}
Reducing $\rho$, if necessary, we can ensure that this inequality holds for all $n$,
and also that $1/\rho$ is an integer. 
Note that $\rho$ may depend on the sequences $(\lambda_k)$ and $(r_k)$,
but it is independent of $K$.
Combining the inequalities \eqref{E:complicated} and  \eqref{E:technical}, we obtain
\[
N_\cA(\rho^n/K)\ge K^{n-1}.
\]
We now make our choice of $K$: it is $K:=\rho^{-n}$. This gives
\[
N_\cA(\rho^{2n})\ge \rho^{-n(n-1)}.
\]
As this holds for each $n\ge2$, we deduce that \eqref{E:cA} holds.
\end{proof}

\section{Proof of the upper bound in Theorem~\ref{T:main}}\label{S:upper}

This time, our starting point is the famous  result of de Branges \cite{dB85} that proved the Bieberbach conjecture.

\begin{theorem}\label{T:Bieberbach}
Let $f\in\cS$, say $f(z)=\sum_{k\ge0}a_kz^k$.
Then $|a_k|\le k$ for all $k$.
\end{theorem}

Special cases of this result had previously been established by various mathematicians for small values of $k$.
Much earlier, Littlewood \cite{Li25} had proved 
the slightly weaker (but much easier) estimate $|a_k|\le ek$ 
 valid for all~$k$.
In fact Littlewood's estimate would do for our purposes.

Again, it is convenient to introduce some notation.
We denote by $\cB$  the family of functions $f(z):=\sum_{k\ge1}a_kz^k$
such that $|a_k|\le k$ for all~$k$.
By virtue of Theorem~\ref{T:Bieberbach}, we have $\cS\subset\cB$,
and so $N_\cS(\delta)\le N_\cB(\delta/2)$.
Thus, to establish the upper bound in Theorem~\ref{T:main},
it suffices to prove the following theorem.

\begin{theorem}\label{T:cB}
Let $d$ be the metric on $\cB$ given by \eqref{E:metric},
where  $(\lambda_j)$ and $(r_j)$ satisfy~\eqref{E:conditions}. Then
\begin{equation}\label{E:cB}
 \log N_\cB(\delta)\lesssim \log^{2+\alpha}(1/\delta)\quad (\delta\to0^+).
\end{equation}
\end{theorem}

For the proof, we need the following approximation lemma.

\begin{lemma}\label{L:cB}
There exists a constant $c>0$ with the following property:
for each $f\in\cB$ and $n\ge1$, there is a polynomial $p\in\cB$
such that $\deg p\le n$ and
\[
d(f,p)\le \exp(-cn^{1/(1+\alpha)}).
\]
\end{lemma}

\begin{proof}
Let $f\in\cB$, say $f(z)=\sum_{k\ge1}a_kz^k$.
Let $n\ge1$ and set $p(z):=\sum_{k=1}^n a_kz^k$.
Clearly $p\in\cB$ and $\deg p\le n$. Also
\[
d(f,p)\le \sum_{j\ge1}\lambda_j \max_{|z|\le r_j}\Bigl|\sum_{k\ge n+1}a_kz^k\Bigr|
\le  \sum_{j\ge1}\Bigl(\lambda_j \sum_{k\ge n+1}kr_j^k\Bigr).
\]
So, to prove the lemma, it suffices to show that
\begin{equation}\label{E:estimate}
\sum_{j\ge1}\Bigl(\lambda_j \sum_{k\ge n+1}kr_j^k\Bigr)
\le \exp(-cn^{1/(1+\alpha)}) \quad(n\ge1),
\end{equation}
where $c>0$ is a constant independent of $n$.

Now, for $r\in(0,1)$, we have
\[
\sum_{k\ge n+1}kr^k=r\frac{d}{dr}\Bigl(\sum_{k\ge n+1}r^k\Bigr)
=r\frac{d}{dr}\Bigl(\frac{r^{n+1}}{1-r}\Bigr)
\le (n+2)\frac{r^{n+1}}{(1-r)^2}.
\]
Therefore
\[
\sum_{j\ge1}\Bigl(\lambda_j \sum_{k\ge n+1}kr_j^k\Bigr)
\le \sum_{j\ge1}\lambda_j (n+2)\frac{r_j^{n+1}}{(1-r_j)^2}.
\]
Recalling that $\lambda_j\asymp\lambda^j$
and $1-r_j\asymp j^{-\alpha}$, we see that there are constants $C_1,c_1>0$, independent of $n$,
such that
\[
\sum_{j\ge1}\lambda_j (n+2)\frac{r_j^{n+1}}{(1-r_j)^2}
\le C_1n \sum_{j\ge1}\lambda^j j^{2\alpha}\exp(-c_1nj^{-\alpha}).
\]
Now
\begin{align*}
\sum_{j\le n^{1/(1+\alpha)}}\lambda^j j^{2\alpha}\exp(-c_1nj^{-\alpha})
&\le \sum_{j\le n^{1/(1+\alpha)}}\lambda^j n^{2\alpha/(1+\alpha)}\exp(-c_1n^{1/(1+\alpha)})\\
&\le \frac{1}{1-\lambda}n^{2\alpha/(1+\alpha)}\exp(-c_1n^{1/(1+\alpha)}).
\end{align*}
Also 
\begin{align*}
\sum_{j> n^{1/(1+\alpha)}}\lambda^j j^{2\alpha}\exp(-c_1nj^{-\alpha})
&\le\sum_{j> n^{1/(1+\alpha)}}\lambda^j j^{2\alpha}\\
&\le \Bigl(\sup_{j\ge1}\lambda^{j/2}j^{2\alpha}\Bigr)\sum_{j> n^{1/(1+\alpha)}}\lambda^{j/2} \\
&\le C_2\frac{(\sqrt{\lambda})^{n^{1/(1+\alpha)}}}{1-\sqrt\lambda},
\end{align*}
where $C_2$ is another constant.
Combining  all these inequalities, we see that \eqref{E:estimate} holds
for any positive constant $c<\min\{c_1,\log(1/\sqrt{\lambda})\}$, at least for all sufficiently large $n$. 
Reducing $c$ further, 
if necessary, we can ensure that \eqref{E:estimate}  holds for all $n$.
\end{proof}

\begin{remark}
The estimate in Lemma~\ref{L:cB} is sharp. Indeed, consider
the so-called Koebe function, namely
$f(z):=z/(1-z)^2=\sum_{k\ge1}kz^k$. 
This belongs to $\cB$, and even to~$\cS$. Also, for every polynomial $p$ of degree $n$, whether it lies in $\cB$ or not, we have
\[
\max_{|z|\le r_j}|f(z)-p(z)|^2
\ge \frac{1}{2\pi r_j}\int_{|z|=r_j}|f(z)-p(z)|^2\,|dz|
\ge\sum_{k\ge n+1}k^2r_j^{2k}\ge r_j^{2n+2}.
\]
It follows that
$d(f,p)\ge \sum_{j\ge1}\lambda_j r_j^{n+1}$.
Retaining just the  term with $j=\lfloor n^{1/(1+\alpha)}\rfloor$,
we see that 
\[
d(f,p)\ge \exp(-cn^{1/(1+\alpha)}),
\]
where $c>0$ is a constant independent of $p$ and $n$.
\end{remark}

\begin{proof}[Proof of Theorem~\ref{T:cB}]
Let $n\ge1$, and let $K$ be a large integer depending on $n$,
to be chosen later.
Consider the collection of polynomials
\[
q(z):=\sum_{k=1}^n k\frac{(s_k+it_k)}{\sqrt{2}K}z^k
\quad(s_k,t_k\in\{-K,-(K-1),\dots,K\},~k=1,\dots,n).
\]
Clearly these polynomials all belong to $\cB$, and there are $(2K+1)^{2n}$
of them.
Given an arbitrary polynomial $p\in\cB$ with $\deg p\le n$,
there is a member $q$ of this collection whose coefficients all
lie within $n/K$ of the corresponding coefficients of $p$,
and consequently
\[
d(p,q)= \sum_{j\ge1}\lambda_j\min\Bigl\{1, \max_{|z|\le r_j}|p(z)-q(z)|\Bigr\}
\le \sum_{j\ge1}\lambda_j \frac{n^2}{K}=Cn^2/K,
\]
where $C$ is a constant independent of $n$.
Combining this inequality with the result of Lemma~\ref{L:cB},
we see that,
if we set  $\delta:=Cn^2/K+\exp(-cn^{1/(1+\alpha)})$,
then  the $\delta$-balls around the polynomials $q$ cover $\cB$.
Hence
\[
N_\cB\Bigl(Cn^2/K+\exp(-cn^{1/(1+\alpha)})\Bigr)\le (2K+1)^{2n}.
\]
We now choose $K$: it is $K:=\lceil Cn^2\exp(cn^{1/(1+\alpha)})\rceil$. This gives
\begin{align*}
N_\cB\Bigl(2\exp(-cn^{1/(1+\alpha)})\Bigr)
&\le (2\lceil Cn^2\exp(cn^{1/(1+\alpha)})\rceil+1)^{2n}\\
&=\exp\Bigl(2n\log\Bigl(2\lceil Cn^2\exp(cn^{1/(1+\alpha)})\rceil+1\Bigr)\Bigr)\\
&\le \exp\Bigl(2n\Bigl(\log(n^2)+cn^{1/(1+\alpha)}+O(1)\Bigr)\Bigr)\\
&\le \exp\Bigl(2cn^{(2+\alpha)/(1+\alpha)}+O(n\log n)\Bigr).
\end{align*}
From this it is straightforward to deduce that \eqref{E:cB} holds.
\end{proof}

\section{Concluding remarks and questions}\label{S:conclusion}

(1) Among several well-known subclasses of  $\cS$,
there is the class $\cC$ of  convex functions, 
namely the set of those $f\in\cS$ such that $f(\DD)$ is a convex set. 
Theorem~\ref{T:main} holds with $\cS$ replaced by $\cC$.
Indeed, the upper bound is obvious. As for the lower bound,
it suffices to repeat the same proof, using the following
result of Goodman~\cite{Go57} in place of Lemma~\ref{L:cA}:
if $f(z):=z+\sum_{k\ge2}a_kz^k$, where $\sum_{k\ge2}k^2|a_k|\le1$, then $f\in\cC$.

(2)  Which, if either, of the bounds in 
\eqref{E:main} is sharp? The bounds are based on the inclusions
$\cA\subset\cS\subset\cB$. The spaces $\cA$ and $\cB$ are easier to handle than $\cS$ because they resemble infinite cartesian products. However,
to answer the question above, we probably need to use finer properties of $\cS$ to determine which of $\cA$ or $\cB$ better approximates $\cS$.

\section*{Acknowledgement}
The authors thank the anonymous referees for their careful reading of the manuscript,
and for their corrections and helpful suggestions.


\begin{thebibliography}{9}

\bibitem{dB85}
L. de Branges, 
A proof of the Bieberbach conjecture, 
\emph{Acta Math.} 154 (1985), 137--152.

\bibitem{Du83}
P. L. Duren,
\emph{Univalent Functions},
Springer-Verlag, New York, 1983. 

\bibitem{Fa14}
K. Falconer,
\emph{Fractal Geometry}, 3rd edition
Wiley, New York, 2014.

\bibitem{Go57}
A. W. Goodman, 
Univalent functions and nonanalytic curves,
\emph{Proc. Amer. Math. Soc.} 8 (1957), 598--601.

\bibitem{Go83}
A. W. Goodman, 
\emph{Univalent Functions}, Vols. I and II,
Mariner Publishing Co., Tampa, FL, 1983. 

\bibitem{KNR12}
T. Kalmes, M. Niess, T. Ransford, 
Examples of quantitative universal approximation, 
In: A. Boivin, J. Mashreghi (Eds.),
\emph{Complex Analysis and Potential Theory}, 2011, Montr\'eal, Canada,
Amer. Math. Soc., Providence, RI, 2012, 77--97.

\bibitem{Li25}
J. E. Littlewood, 
On inequalities in the theory of functions, 
\emph{Proc. London Math. Soc.} 23 (1925), 481--519.

\bibitem{Po75}
C. Pommerenke, 
\emph{Univalent Functions},
Vandenhoeck \& Ruprecht, G\"ottingen, 1975.

\bibitem{Sc75}
G. Schober, 
\emph{Univalent Functions--Selected Topics},
Springer-Verlag, Berlin-New York, 1975.

\end{thebibliography}
\end{document}